\documentclass[11pt]{amsart}

\usepackage{amssymb,amsthm}
\usepackage{amsmath}

\newtheorem{theorem}{Theorem}
\newtheorem{corollary}[theorem]{Corollary}

\newtheorem{lemma}[theorem]{Lemma}

\newtheorem*{theorem*}{Theorem}

\newtheorem{algorithm}[theorem]{Algorithm}

\def\irr#1{{\rm  Irr}(#1)}

\def\phi{{\varphi}}

\newcommand{\cl}[1]{{\rm Cl} (#1)}

\begin{document}

\title[Super character theories]{Groups with exactly two supercharacter theories}
\author[Burkett]{Shawn Burkett}
\author[Lamar]{Jonathan Lamar}
\author[Lewis]{Mark L. Lewis}
\author[Wynn]{Casey Wynn}

\address{Department of Mathematics, University of Colorado, Boulder, CO 80309}
\email{shawn.burkett@colorado.edu}

\address{Department of Mathematics, University of Colorado, Boulder, CO 80309}
\email{jonathan.lamar@colorado.edu}

\address{Department of Mathematical Sciences, Kent State University, Kent, OH 44242}
\email{lewis@math.kent.edu}

\address{Department of Mathematical Sciences, Kent State University, Kent, OH 44242}
\email{cwynn@math.kent.edu}


\begin{abstract}
In this paper, we classify those finite groups with exactly two supercharacter theories.  We show that the solvable groups with two supercharacter theories are $\mathbb{Z}_3$ and $S_3$.  We also show that the only nonsolvable group with two supercharacter theories is ${\rm Sp} (6,2)$.
\end{abstract}

\maketitle

\bigskip
\section{Introduction}

In \cite{DiIs}, Diaconis and Isaacs define supercharacter theories for finite groups, formalizing previous approaches to understanding the characters of $U_{n}(\mathbb{F}_{q})$ and other algebra groups.  More recently, Johnson and Humphries showed a strong connection exists between supercharacter theories and central subalgebras of Schur rings over finite groups \cite{HuJo}.  Using this connection, as well as results on Schur rings by Leung and Man \cite{LeMa} and Tamaschke \cite{Tam}, Anders Hendrickson provides constructions for supercharacter theories and classifies all supercharacter theories for cyclic groups \cite{AHen}.  We will show that the only groups with exactly two supercharacter theories, in particular the trivial supercharacter theories, are $S_3$, $\mathbb{Z}_3$, and  ${\rm Sp}(6,2)$.

Throughout this paper $G$ will be a finite group.  We write $\irr G$ for the set of irreducible characters of $G$ and $\cl G$ for the set of conjugacy classes of $G$.  A {\it supercharacter theory} of $G$ is defined in \cite{DiIs} to be a pair $(\mathcal {X}, \mathcal {K})$ where $\mathcal {X}$ is a partition of $\irr G$ and $\mathcal {K}$ is a partition of $G$ that satisfy the following (1) $|\mathcal {X}| = |\mathcal {K}|$, (2) for each $X \in \mathcal {X}$, there exists a character $\chi_X$ whose irreducible constituents lie in $X$ such that $\chi_X$ is constant on the members of $\mathcal {K}$, and (3) $\{ 1 \} \in \mathcal {K}$.  This definition implies that each ``supercharacter'' $\chi_{X}$ is a positive integer constant multiple of $\sum_{\phi \in X} \phi (1)\cdot\phi$ and that $\{ 1_{G} \} \in \mathcal{X}$.  In \cite{DiIs}, they show that if we know a partition $\mathcal X$ of $\irr G$ yields a supercharacter theory, then the corresponding the partition $\mathcal {K}$ of $G$ is uniquely determined, so often we will only provide the partition $\mathcal X$ for a supercharacter theory.  When $G$ is a group, we will write $\rm Sup (G)$ to denote the set of all supercharacter theories for $G$.

On page 2363 of \cite{DiIs}, they mention that there exist groups having only two supercharacter theories, particularly $S_3$.  However, they give no sense of how rare these groups are.  In this note, we consider groups with only two supercharacter theories, and we will show that in fact there are only two solvable groups and one nonabelian simple group that have only two supercharacter theories.

The third and fourth authors would like to thank Nat Thiem for several useful conversations regarding this work, and for suggesting this problem to the first and second authors.

\section{The solvable case}

Every group has the following two supercharacter theories: $m(G)$ where $ \irr G$ is partitioned into singleton sets and $G$ is partitioned into its conjugacy classes, and $M(G) = (\{ \{1_G\}, \irr G - \{1_G\} \}, \left\{\left\{1\right\}, G-\left\{1\right\}\right\})$.  These two supercharacter theories are the same if and only if $\rho_G - 1_G$ is either $0$ or irreducible where $\rho_G$ is the regular character for $G$.  That occurs only when $G$ is the trivial group or when $G = \mathbb{Z}_2$.  For all other groups, these two supercharacter theories will be distinct.

Notice that if $G$ is either $\mathbb{Z}_3$ or $S_3$, then it is not difficult to see since $\rho_G - 1_G$ has only two irreducible constituents that $m(G)$ and $M(G)$ are the only possible supercharacter theories.  We now show that these are the only solvable groups with two supercharacter theories.  Following the usual convention, we say that $G$ is a {\it real group} if all the characters in $\irr G$ are real valued.

\begin{theorem} \label{one}
If $G$ is a group satisfying $|\rm {Sup} (G)| = 2$, then either $G \cong S_3$ or $G$ is a real simple group.  In particular, if $G$ is solvable and $|\rm {Sup} (G)| = 2$, then $G$ is either $\mathbb{Z}_3$ or $S_3$.
\end{theorem}

\begin{proof}
	Now, suppose $G$ is group that has only the two supercharacter theories: $m(G)$ and $M(G)$.  Suppose $N$ is a normal subgroup of $G$ so that $1 < N < G$.  Let $\rho_{G/N}$ be the inflation of the regular character of $G/N$ to $G$.  It is not difficult to see that $\mathcal{X} = \{ \{ 1_G \}, \irr {G/N} - \{1_G\}, \irr G - \irr {G/N} \}$ yields a supercharacter theory for $G$.  Under the assumption that $G$ only has two supercharacter theories, this one must be either $m(G)$ or $M(G)$.  It is not difficult to see that it cannot be $M(G)$, so it must be $m(G)$.  It follows that $\rho_{G/N} - 1_G$ and $\rho_G - \rho_{G/N}$ must be multiples of irreducible characters.  This implies that $\irr {G/N}$ has only one nonprincipal irreducible character, and so, $G/N \cong \mathbb{Z}_2$.  It also follows that $\irr {G \mid N}$ has a unique irreducible character.  This implies that $G/N$ acts transitively on $\irr N - \{ 1_N \}$.  Since $G/N$ has size $2$, it follows that $N$ has size $3$.  We conclude that $G \cong S_3$.

	We now assume that $G$ has no proper nontrivial normal subgroups,  i.e., $G$ is a simple group.  Let $\mathcal {R}$ be the supercharacter theory obtained by taking the orbits of complex conjugation on $\irr G$, as on page 2360 of \cite{DiIs}.  If $\mathcal {R} = M(G)$, then it follows that $\irr G - \{ 1_G \}$ consists of two characters that are complex conjugates, and so, we conclude that $G \cong \mathbb{Z}_3$.  (Recall we are assuming that $M(G) \ne m(G)$.)  If $\mathcal {R} = m(G)$, then all characters in $\irr G$ are real valued.  If $G$ is solvable, then $G$ is cyclic.  The only real cyclic group is $Z_2$, so this implies $G \cong \mathbb{Z}_2$.  We conclude if $G$ is solvable, then $G$ is either $\mathbb{Z}_3$ or $S_3$.
\end{proof}

\section{The nonsolvable case}

A group $G$ is said to be \emph{rational} if all the irreducible characters are rational valued.  The nonabelian simple groups that are rational have been classified by Feit and Seitz in \cite{FeSe}.   We note that their result relies on the classification of finite simple groups, and thus, our result also relies on the classification of finite simple groups.

\begin{theorem}
If $G$ is nonsolvable and $|\rm {Sup} (G)| = 2$, then $G$ is ${\rm Sp} (6,2)$.
\end{theorem}

\begin{proof}
In light of Theorem \ref{one}, if $G$ is nonsolvable, then $G$ must be a real, simple group.  Let $\mathcal Q$ be the supercharacter theory by taking the orbits of the Galois group of $\mathbb{Q}(G)/\mathbb{Q}$ on $\irr G$, again as on page 2360 of \cite{DiIs}.  It is obvious that $\mathcal Q$ is not $M(G)$, so $\mathcal Q$ must be $m(G)$.  This implies that $G$ is a rational group.  By Corollary B.1 of \cite{FeSe}, the only nonabelian simple, rational groups are ${\rm Sp} (6,2)$ and ${\rm SO}^+ (8,2)$.  In addition, it is known that ${\rm SO}^+ (8,2)$ has a nontrivial outer automorphism $\sigma$ and that $\sigma$ has a nontrivial action on $\irr G$.  If $\mathcal S$ is the supercharacter theory obtained by taking the orbits of $\sigma$ on $\irr G$, we see that $\mathcal S$ is neither $M(G)$ nor $m(G)$.  We conclude that $G$ is not ${\rm SO}^+ (8,2)$, and hence, ${\rm Sp} (6,2)$ is the only remaining possibility.
\end{proof}

At this point, we have shown that if there exists a nonsolvable group with only two supercharacter theories, then that nonsolvable group must be ${\rm Sp} (6,2)$.  We now work to show that ${\rm Sp} (6,2)$ has only two supercharacter theories, and we now develop the tools to do this.  In particular, we will find a means of showing for an arbitrary group $G$ that a subset $X \subseteq \irr G$ cannot be a subset in the character partition of any supercharacter theory of $G$.  These methods are adapted from an algorithm in \cite{AHen}.  We use this idea to eliminate all proper subsets of $\irr {{\rm Sp} (6,2)}$ that contain more than one character from occurring in the character partition of any supercharacter theory of $G$.

For any subset $X\subseteq\irr G$, the \emph{Wedderburn sum} corresponding to $X$, denoted $\sigma_X$, is the character	$\sigma_X = \sum_{\chi\in X}\chi(1)\chi$.  A \emph{partial partition} of $\irr G$ is a set $\mathcal{X} = \{ X_1, \dots, X_n \}$ where each $X_i$ is a nonempty subset of $\irr G$ and $X_i \cap X_j = \emptyset$ when $i \ne j$.  If $\bigcup_{i=1}^n X_i = \irr G$, then $\mathcal X$ is a \emph{partition} of $\irr G$.  If $\mathcal {X}$ and $\mathcal {Y}$ are partial partitions, then $\mathcal {X} \subseteq \mathcal {Y}$ if $X_i \in \mathcal {Y}$ for every $X_i \in \mathcal {X}$.  We say that a partition $\mathcal {Y}$ is a \emph{refinement} of the partition $\mathcal {X}$ if every set $X_i \in \mathcal {X}$ is a union of sets in $\mathcal {Y}$, and we write $\mathcal {Y} \preceq \mathcal {X}$ when $\mathcal {Y}$ is a refinement of $\mathcal {X}$.


We write $\mathrm {cf}(G)$ for the set of complex valued class functions on $G$.  It is well known that $\mathrm {cf} (G)$ is an algebra and $\irr G$ is a basis.  Let $\mathcal{X}$ be a partial partition of $\irr G$ and let $\mathcal{A}$ be the subalgebra of $\mathrm{cf}(G)$ generated by $\{ \sigma_X \mid X\in\mathcal{X} \}$.  We define an equivalence relation $\sim$ on $\irr G$ as follows: $\chi \sim \psi$ if and only if the coefficients of $\chi(1)\chi$ and $\psi(1)\psi$ in $a$ are the same for all $a \in \mathcal{A}$.  Then the \emph{filtration} given by $\mathcal{X}$, denoted $\mathcal{F} (\mathcal{X})$, is the partition of $\irr G$ into equivalence classes with respect to this relation.
%
The algebra $\mathcal{A}$ lies in the span of $\{\sigma_F \mid F \in \mathcal {F} (\mathcal{X})\}$, so each set in $\mathcal{X}$ is a union of sets in $\mathcal{F} (\mathcal{X})$.  When we assume $\mathcal{X}$ to be a subset of a partition which yields a supercharacter theory, we may draw two further conclusions.


\begin{lemma}\label{lem:A.4star}
Let $\mathcal{X}$ be a partial partition of $\irr G$.  Suppose there exists a supercharacter theory $C = (\mathcal{Y}, \mathcal{L})$ such that $\mathcal{X} \subseteq \mathcal{Y}$.  Then $\mathcal {Y}$ is a refinement of $\mathcal{F}(\mathcal{X})$ and $\mathcal{X} \subseteq \mathcal{F}(\mathcal{X})$.
\end{lemma}

\begin{proof}
	Let $\mathcal{A}$ be the subalgebra of $\mathrm{cf}(G)$ generated by $\{\sigma_X \mid X \in \mathcal{X}\}$.  Then
	\[
		\mathcal{A} = \langle\sigma_X \mid X\in\mathcal{X} \rangle \subseteq \langle \sigma_Y \mid Y\in\mathcal{Y} \rangle = \mathrm{Span}\big\{\sigma_Y \mid Y\in\mathcal{Y}\big\},
	\]
	the last equality following from \cite[Lemma 2.1]{And14}.  Let $Y \in \mathcal{Y}$ and let $\chi,\psi\in Y$; then for all $a\in\mathcal{A}$, the coefficients of $\chi(1)\chi$ and $\psi(1)\psi$ are the same because $a \in \mathrm{Span} \{ \sigma_Z \mid Z\in\mathcal{Y} \}$.  Then by definition of filtration, $\chi$ and $\psi$ lie in the same equivalence class of $\sim$, so they lie in the same set in $\mathcal{F}(\mathcal{X})$, so it follows that $\mathcal{Y} \preceq \mathcal{F}(\mathcal{X})$.
	
Now, for each set $X$ in $\mathcal{X}$, because $X$ is a union of sets in $\mathcal{F}(\mathcal{X})$, choose a set $F \in \mathcal{F}(\mathcal{X})$ such that $F\subseteq X$.  Then since $\mathcal{Y} \preceq \mathcal{F}(\mathcal{X})$, the set $F$ must be a union of sets in $\mathcal{Y}$.  Since the only set of $\mathcal{Y}$ overlapping $X$ is $X$ itself, we have $X\subseteq F$.  Hence $X = F \in \mathcal{F}(\mathcal{X})$.  We conclude that $\mathcal{X} \subseteq \mathcal{F}(\mathcal{X})$, and the proof is complete.
\end{proof}

Thus, the filtration of a partial partition $\mathcal{X}$ is coarser than the character partition of every supercharacter theory which contains $\mathcal{X}$.  By applying this observation to a single set of characters $X \subseteq \irr G$, we derive a necessary condition for the existence of a supercharacter theory in which $X$ is a supercharacter.  Let $G$ be a group and let $X$ be a subset of $\irr G$.  We say $X$ is \emph{good} if $X \in \mathcal{F}(\{X\})$; otherwise we say $X$ is \emph{bad}.

\begin{corollary}\label{cor:SCimpliesGood}
	Let $G$ be a group with exactly $n$ irreducible characters.
	\begin{enumerate}
		\item Let $C = (\mathcal {Y}, \mathcal {L}) \in\mathrm{Sup}(G)$ and let $X \in \mathcal {Y}$.  Then $X$ is good.
		
		\item Let $X$ be a subset of $\irr G$.  Then $X$ is bad if and only if there exist characters $\chi, \psi \in X$ and an integer $j \in \{2,\ldots,n\}$ such that the coefficients of $\chi(1)\chi$ and $\psi(1)\psi$ in $\sigma_X^j$ differ.
	\end{enumerate}
\end{corollary}

\begin{proof}
	Let $C=(\mathcal{Y},\mathcal{L})$ and let $X \in \mathcal {Y}$.  Then $\{ X \} \subseteq \mathcal{Y}$, so by Lemma \ref{lem:A.4star}, we know $\{X\} \subseteq \mathcal{F}(\{X\})$.  By definition $X$ is good, proving conclusion (1).
	
	Let $X$ be a subset of $\irr G$, and suppose $X$ is bad.  Now, $X$ is a union of sets in $\mathcal{F}(\{X\})$, but $X$ is not itself a set in $\mathcal{F}(\{X\})$, because it is bad.  Therefore, there exist characters $\chi, \psi \in X$ that lie in different sets in $\mathcal{F}(\{X\})$; so there exists some element $a \in \langle \sigma_X \rangle$ such that the coefficients of $\chi(1)\chi$ and $\psi(1)\psi$ in $a$ are different.  Now, because $\langle \sigma_X \rangle \subseteq \mathrm{cf} (G)$ is at most $n$ dimensional, it follows that $\langle\sigma_X\rangle$ is spanned by $\{\sigma_X,\sigma_X^2,\ldots,\sigma_X^n\}$.  If the coefficient of $\chi(1)\chi$ in $\sigma_X^j$ equals that of $\psi(1)\psi$ for all $j\in\{1, \ldots, n\}$, it would follow that the coefficients of $\chi(1)\chi$ and $\psi(1)\psi$ would be identical for every element $a$ in $\langle \sigma_X \rangle$.  This implies that $\chi \sim \psi$, and thus, $\chi$ and $\psi$ lie in the same set of $\mathcal{F} (\{X \})$, a contradiction.  Hence, there exists some integer $j \in \{1, \ldots, n\}$ such that $\chi(1)\chi$ and $\psi(1)\psi$ have different coefficients in $\sigma_X^j$, and $j$ is certainly not $1$.
	
	On the other hand, if there exist different characters $\chi, \psi \in X$ and an integer $j \in \{ 2, \ldots, n \}$ such that $\chi(1)\chi$ and $\psi(1)\psi$ have different coefficients in $\sigma_X^j$, then $\chi$ and $\psi$ lie in different sets in $\mathcal{F}(\{X\})$, so $X$ is not in $\mathcal{F}(\{X\})$.  This implies that $X$ is bad, and completes the proof of conclusion (2).
\end{proof}

Note that Corollary \ref{cor:SCimpliesGood} (1) implies that if there is a supercharacter theory $C$ for $G$ so that $X$ is a set in the partition of $\irr G$, then $X$ must be good.  However the converse of Corollary \ref{cor:SCimpliesGood} (1) is not true in general.  For example, let $X$ be the set of all irreducible characters of $S_7$ of degree $14$. Although the algorithm below can be used to show that $X$ is good, computer computations indicate that $X$ is not a part of any supercharacter theory of $S_7$. 

By Corollary \ref{cor:SCimpliesGood} (1), it will follow that ${\rm Sp} (6,2)$ has only the two trivial supercharacter theories once we show that the only good subsets of $\irr G \setminus \{1_G\}$ are the whole set and its singleton subsets.  By Corollary \ref{cor:SCimpliesGood} (2), we need only to check powers of Wedderburn sums against their underlying sets' filtrations to determine whether a subset is good.  The following algorithm does exactly this.

\begin{algorithm}\label{alg:isgood}
	Given any group $G$ and any subset $X = \{ \chi_i \mid i \in I\}$ of $\irr G$, this algorithm determines if $X$ is good.  Suppose $G$ has precisely $n$ conjugacy classes with representatives $g_1,\ldots,g_n$ and irreducible characters $\chi_1,\ldots,\chi_n$, and let $T$ be the character table, represented as an $n\times n$ matrix whose $i,j$ entry is $\chi_i(g_j)$.
	\begin{enumerate}
		\item Form the Wedderburn sum $\sigma_X$ corresponding to $X$.  Then with respect to the conjugacy class identifier basis, $\sigma_X$ takes the form
		\[
			\sigma_X = \sum_{j=1}^n\sigma_X(g_j)\delta_{[g_j]},
		\]
		where $\delta_{[g_j]}$ is the indicator function of the conjugacy class containing $g_j$.
		\item For each $k\in\{2,\ldots,n\}$, do the following.
		\begin{enumerate}
			\item Consider the $k$th tensor power of $\sigma_X$; with respect to the above basis, we have
			\[
				\sigma_X^k = \sum_{j=1}^n\sigma_X(g_j)^k\delta_{[g_j]}.
			\]
			With respect to the irreducible character basis, write
			\[
				\sigma_X^k = \sum_{i=1}^nc_{i,k}\chi_i(1)\chi_i.
			\]
			\item Set $m=\min(I)$.  If there exists $i\in I$ for which $c_{i,k}\neq c_{m,k}$, return bad.
			\item Otherwise, continue.
		\end{enumerate}
		\item Return good.
	\end{enumerate}
\end{algorithm}

Now, we may convert this algorithm into SAGE code.  Let \verb|T| be a SAGE matrix representing the character table of $\mathrm{Sp}(6,2)$ as in Algorithm \ref{alg:isgood}; such a matrix can be computed from GAP or MAGMA.  Let \verb|S| be the inverse of \verb|T|.  The following function is a SAGE implementation of Algorithm \ref{alg:isgood}.  Note that subsets of $\mathrm{Irr}(G)$ are represented by subsets of $\{0,1,\ldots,29\}$, where the ordering on $\mathrm{Irr}(G)$ is determined by GAP.

\bigskip

\begin{verbatim}
def is_good(X):
  sigma = matrix([0]*30)
  for i in X:
    sigma = sigma + matrix([T[i,0]*k for k in T[i]])
  power_of_sigma = sigma
  for i in range(2,31):
    power_of_sigma = power_of_sigma.elementwise_product(sigma)
    new_basis = power_of_sigma*S
    for j in X[1:]:
      if new_basis[0,j]/T[j,0] != new_basis[0,X[0]]/T[X[0],0]:
        return False
  return True
\end{verbatim}

\bigskip

We then record the output of this function for each subset of $\{ 0, 1, \ldots, 29 \}$.  In order to reduce runtime, we recorded the output only on subsets of $\{ 0, 1, \ldots, 29 \}$ which we did not already know to be good, i.e., proper subsets of $\{ 1, \ldots, 29 \}$ which are not singletons, and the algorithm returned false for every proper subset of $\{ 0, 1, \ldots, 29 \}$ which is not a singleton.  Since none of these subsets were good, it follows that $\mathrm{Sp}(6,2)$ has no supercharacter theory where the partition of $\irr G - \{ 1_G \}$ does not consist of a single set or of singleton sets.  Thus, the minimal and maximal supercharacter theories are the only supercharacter theories of $\mathrm{Sp} (6,2)$.  We conclude that $\mathrm{Sp} (6,2)$ has exactly two supercharacter theories.  Note that applying this method to other groups will not find nontrivial supercharacter theories, however it can negatively answer the question of whether given sets of irreducible characters can occur in the partition of characters for some supercharacter theory.

\end{document}